\RequirePackage{ifpdf}
\ifpdf 
\documentclass[pdftex]{sigma}
\else
\documentclass{sigma}
\fi

\def\phi{{\varphi}}

\def\T*M{\mathrm{T}^*\mathrm{M}}

\let\na\nabla

\DeclareMathAlphabet{\doba}{U}{msb}{m}{n}

\def\T{\mathrm{T}}
\def\Vol{{\mathop{\rm Vol}}}
\def\Ric{{\mathop{\rm Ric}}}

\def\Spin{{\mathop{\rm Spin}}}
\def\SO{{\mathop{\rm SO}}}

\def\aire{\mathrm{Area}(M^2,g)}
\def\Vol{{\mathop{\rm Vol}}}

\def\Ric{{\mathop{\rm Ric}}}

\def\Spin{{\mathop{\rm Spin}}}
\def\SO{{\mathop{\rm SO}}}

\let\<\langle
\let\>\rangle

\let\<\langle
\let\>\rangle

\long\def\komment#1{}

\begin{document}

\allowdisplaybreaks

\renewcommand{\PaperNumber}{119}

\FirstPageHeading

\renewcommand{\thefootnote}{$\star$}

\ArticleName{Branson's $\boldsymbol{Q}$-curvature in Riemannian\\ and Spin Geometry\footnote{This paper is a
contribution to the Proceedings of the 2007 Midwest
Geometry Conference in honor of Thomas~P.\ Branson. The full collection is available at
\href{http://www.emis.de/journals/SIGMA/MGC2007.html}{http://www.emis.de/journals/SIGMA/MGC2007.html}}}

\ShortArticleName{Branson's $Q$-curvature in Riemannian  and Spin Geometry}

\Author{Oussama HIJAZI and  Simon RAULOT}

\AuthorNameForHeading{O. Hijazi and S. Raulot}

\Address{Institut {\'E}lie Cartan Nancy, Nancy-Universit{\'e}, CNRS, INRIA,\\
Boulevard des Aiguillettes B.P. 239 F-54506 Vand{\oe}uvre-l{\`e}s-Nancy Cedex, France}
\Email{\href{mailto:hijazi@iecn.u-nancy.fr}{hijazi@iecn.u-nancy.fr}, \href{mailto:Simon.Raulot@unine.ch}{Simon.Raulot@unine.ch}}

\ArticleDates{Received August 25, 2007, in f\/inal form November
29, 2007; Published online December 11, 2007}

\Abstract{On a closed $n$-dimensional manifold, $n\ge 5$, we compare the
  three basic conformally covariant operators: the Paneitz--Branson, the
  Yamabe and the Dirac operator (if~the manifold is spin) through their
  f\/irst eigenvalues. On a closed $4$-dimensional Riemannian manifold, we give a lower bound for the square of the f\/irst eigenvalue of the Yamabe ope\-rator in terms of the total Branson's $Q$-curvature. As a consequence, if the manifold is spin, we relate the f\/irst eigenvalue of the Dirac operator to the total Branson's $Q$-curvature.    Equality cases are also characterized.}


\Classification{53C20; 53C27; 58J50}

\rightline{\it To the memory of Tom Branson}

\vspace{2mm}

\begin{flushright}
\begin{minipage}{14cm}
\it Tom has deeply influenced my life. With him,
I learned to push the limits of what would concretely mean
to have a clear and deep thinking, to take a huge distance
from things and events so that the essence could be touched.\\[1mm]
\null \hfill Oussama Hijazi
\end{minipage}
\end{flushright}

\section{Introduction}

The scalar curvature function behaves nicely under a conformal change of the metric. The Yamabe operator relates
the scalar curvatures associated with two metrics in the same conformal class. From the conformal point of view dimension $2$ is special and the scalar curvature corresponds to (twice) the Gauss curvature. The problem of conformally deforming the scalar curvature into a constant  is known as the Yamabe problem, it  has been intensively studied in the seventies and solved in the beginning eighties.

 On a spin compact manifold it is an important fact, through the Schr{\"o}dinger--Lichnerowicz formula, that the scalar curvature is closely related to the eigenvalues of the Dirac operator and its sign inf\/luences the topology of the manifold.

The conformal behavior of the scalar curvature and that of the Dirac operator allowed to establish lower bounds for the square of the eigenvalues of the Dirac operator in terms of scalar conformal invariants, the Yamabe invariant in dimension at least $3$ and the total scalar curvature in dimension $2$ (see \cite{hijazi:86, hijazi:91, baer:92}).

 Given this inequality for surfaces, one might ask whether such integral inequalities could be true in any dimension. It is shown in \cite{ammann-baer} that this could not be true with the total scalar curvature. However one could expect such integral lower bounds associated with  other curvatures.

 The Branson $Q$-curvature is a scalar function which shares with the scalar curvature inte\-res\-ting conformal behavior. The operator which relates the $Q$-curvatures associated with two conformally related metrics is the Paneitz--Bransonoperator, a $4$-th order dif\/ferential operator. Another curvature function which is of special interest from the conformal aspect is the $\sigma_k$-curvature, which is the $k$-th symmetric function associated with the eigenvalues of the Schouten tensor (the tensor which measures the defect of the conformal invariance of the Riemann tensor).

 Taking into account the analogy between the scalar curvature and the $Q$-curvature, some natural questions could be asked: what would be the role of the $Q$-curvature in Spin Geomet\-ry and is there any relation between the spectra of the three natural conformally covariant dif\/ferential operators : Dirac, Yamabe and Branson--Paneitz?

 In this paper, a f\/irst answer is given for $n$-dimensional closed Riemannian  manifolds. For $n=4$,  we establish a new lower bound for the square of the f\/irst eigenvalue of the Yamabe operator in terms of the total Branson's $Q$-curvature (see Theorem \ref{main4}). For $n\ge 5$, we show that, up to a constant, the square of the f\/irst eigenvalue of the Yamabe operator  is at least equal to that of  the Paneitz--Branson operator (see Theorem~\ref{main}).

 In case the manifold is spin, we make use of what is {\it called}  the
 Hijazi inequality to relate the f\/irst eigenvalue of the Dirac operator
 to the total Branson's $Q$-curvature, if $n=4$ (see Corollary~\ref{corr4}),  and to the f\/irst eigenvalue of the  Paneitz--Branson
 operator,  if $n\ge 5$ (see Corollary~\ref{corr1}).

 The key classical argument in Spin Geometry (see \cite{hijazi:86}) is to consider on a Riemannian manifold a special metric in the conformal class associated with an appropriate choice of the conformal factor, namely an eigenfunction associated with the f\/irst eigenvalue of the Yamabe operator.

\section{Natural geometric operators\\ in conformal Riemannian geometry}\label{ngocg}

Consider a compact Riemannian manifold $(M^n,g)$ and let
$[g]=\{g_u:=e^{2u}g\,/\,u\in C^{\infty}(M)\}$ be the conformal class of the
metric $g$. A class of dif\/ferential operators of particular interest in
Riemannian Geometry is that of conformally covariant operators. If $A:=A_g$ is a formally self-adjoint geometric dif\/ferential operator acting on functions (or on sections of vector bundles) over $(M^n,g)$ and $A_u:=A_{g_u}$ then $A$ is said to be conformally covariant of bidegree $(a,b)\in\mathbb{R}^2$ if and only if
\begin{gather*}
A_u(\cdot)=e^{-bu} A( e^{au}\cdot).
\end{gather*}
We now give some relevant examples of such operators.

\subsection{The Yamabe  operator}\label{laplacienconforme}


In dimension $n=2$, the Laplacian $\Delta:=\delta d$ acting on smooth functions is a conformally covariant dif\/ferential operator of bidegree $(0,2)$ since it satisf\/ies $\Delta_u=e^{-2u} \Delta$. It is interesting to note that we have the {\it Gauss curvature equation}:
\begin{gather}\label{equationgauss}
\Delta u + K = K_u  e^{2u},
\end{gather}
where $K_u$ is the Gauss curvature of $(M^2,g_u)$. Using this formula, we can easily conclude that:
\begin{gather*}
\int_{M} K dv=\int_M K_u dv_u,
\end{gather*}
is a conformal invariant of the surface $M$ equipped with the conformal class of $g$. In fact, it is a topological invariant due to the Gauss--Bonnet formula:
\begin{gather}\label{gausssurface}
2\pi\chi(M^2)=\int_{M}Kdv,
\end{gather}
where $\chi(M^2)$ is the Euler--Poincar{\'e} characteristic class of $M$.

In dimension $n\geq 3$, the Yamabe operator (or the conformal Laplacian):
\begin{gather*}
L:=4\frac{n-1}{n-2}\Delta+R,
\end{gather*}
where $R$ is the scalar curvature of $(M^n,g)$, is conformally covariant of bidegree $(\frac{n-2}{2},\frac{n+2}{2})$. Indeed, we have the following relation:
\begin{gather*}
L_u f=e^{-\frac{n+2}{2}u} L(e^{\frac{n-2}{2}u}f),
\end{gather*}
for all $f\in C^{\infty}(M)$. It is important to note that this operator relates the scalar curvatures of the manifold $M$ associated with two metrics in the same conformal class. Indeed, we have:
\begin{gather*}
Lu=R_u\, u^{\frac{n+2}{n-2}}
\end{gather*}
for $g_u=u^{\frac{4}{n-2}}\in[g]$, where $u$ is a smooth positive function.

\subsection[The Branson-Paneitz operator]{The Paneitz--Branson operator}\label{paneitz}


In dimension $n\geq 3$, the Paneitz--Branson operator, def\/ined by
\begin{gather*}
P:=\Delta^2+\delta(\alpha_nR\,{\rm{Id}}+\beta_n \Ric)d+\frac{n-4}{2}Q
\end{gather*}
is a self-ajdoint elliptic fourth-order conformally covariant dif\/ferential operator of bidegree $(\frac{n-4}{2},\frac{n+4}{2})$, i.e.,
\begin{gather*}
P_u f=e^{-\frac{n+4}{2}u}\, P(e^{\frac{n-4}{2}u}f),
\end{gather*}
for all $f\in C^{\infty}(M)$. Here $\alpha_n=\frac{(n-2)^2+4}{2(n-1)(n-2)}$, $\beta_n=-\frac{4}{n-2}$, $\Ric$ is the Ricci tensor and $Q$ is the Branson $Q$-curvature of $(M^n,g)$ def\/ined by:
\begin{gather}\label{qcurvature}
Q=\frac{n}{8(n-1)^2}R^2-2|S|^2+\frac{1}{2(n-1)}\Delta R,
\end{gather}
where:
\begin{gather*}
S_{ij}  =  \frac{1}{n-2}A_{ij}=\frac{1}{n-2}\left(\Ric_{ij}-\frac{1}{2(n-1)}R g_{ij}\right).
\end{gather*}
is the Schouten tensor. The $Q$-curvature together with the Paneitz--Branson operator share a~similar conformal behavior as that of the scalar curvature and the Yamabe operator. Indeed, if $n\geq 5$ and $g_u=u^{\frac{4}{n-4}}$ we have:
\begin{gather}\label{paneitzconforme}
Pu=\frac{n-4}{2}\, Q_u\, u^{\frac{n+4}{n-4}},
\end{gather}
 where $Q_u$ is the $Q$-curvature of $(M^n, g_u)$. On a $4$-dimensional manifold, the Paneitz--Branson operator $P$ and the $Q$-curvature are analogues of the Laplacian and the Gauss curvature on  $2$-dimensional manifolds. In fact, we have the following $Q$-{\it curvature equation}:
\begin{gather*}
Pu+Q=Q_u \, e^{4u}
\end{gather*}
for $g_u=e^{2u}g\in[g]$ (compare with (\ref{equationgauss})). Again, using this equation, we can easily deduce that:
\begin{gather*}
\int_{M}Q\, dv=\int_M Q_u\, dv_u
\end{gather*}
is a conformal invariant. Another way to obtain the conformal covariance of the total $Q$-curvature functional comes from the Chern--Gauss--Bonnet formula for $4$-dimensional manifolds:
\begin{gather}\label{gaussbonnet}
16\pi^2\chi(M^4)=\int_{M}\left(\frac{1}{2}|W|^2+\frac{1}{12}R^2-|E|^2\right)dv,
\end{gather}
 where $E:=\Ric-(1/n)R g$ is the Einstein tensor of $(M^4,g)$. Thus using (\ref{qcurvature}) for $n=4$, we obtain:
\begin{gather}\label{CGB_Q}
16\pi^2\chi(M^4)=\frac{1}{2}\int_{M}|W|^2dv+2\int_{M}Qdv.
\end{gather}
Since the expression $|W|^2dv$ is a pointwise conformal invariant, we conclude that the total $Q$-curvature is a conformal invariant.

\section{Spin geometry, the Dirac operator\\ and classical eigenvalue estimates}\label{dirac}

For convenience we brief\/ly recall some standard facts about Riemannian Spin Geometry (see~\cite{lawson.michelson:89} or~\cite{fried}) and establish with some details the fact  that, as the Yamabe and Paneitz--Branson operators, the Dirac operator $D$ acting on smooth sections of the spinor bundle is a conformal covariant operator of bidegree $(\frac{n-1}{2},\frac{n+1}{2})$. We then give a proof of classical conformal eigenvalue lower bounds on the spectrum of the Dirac operator.

We consider a closed compact Riemannian manifold $(M^n,g)$ equipped with a spin structure, which is a topological restriction corresponding to an orientability condition of order two. Thanks to this structure, one can construct a complex vector bundle $\Sigma_gM:=\Sigma M$ (the bundle of complex spinors) of rank $2^{[n/2]}$ over $M$. A smooth section $\psi\in\Gamma(\Sigma M)$ of this vector bundle is called a spinor f\/ield. Note that this vector bundle depends on the Riemannian metric. The spinor bundle $\Sigma M$ is endowed with
\begin{enumerate}\itemsep=0pt
\item[1)] Clif\/ford multiplication, that is an action of the tangent bundle on spinor f\/ields:
\[
\begin{array}{cll}
{\rm{T}}M\otimes \Sigma M & \longrightarrow & \Sigma M\\
X\otimes \psi & \longmapsto & X\cdot\psi,
\end{array}
\]

\item[2)] the natural spinorial Levi-Civita connection $\nabla$ acting on smooth spinor f\/ields correspon\-ding to the Levi-Civita connection  (also denoted by $\nabla$) and satisfying:
\[
\nabla_X(Y\cdot\varphi)=(\nabla_X Y)\cdot\varphi+Y\cdot\nabla_X\varphi,
\]

\item[3)] a natural Hermitian scalar product $\<\ ,\ \>$ such that:
\begin{gather*}
\<X\cdot\psi,X\cdot\varphi\>=g(X,X)\<\psi,\varphi\>
\end{gather*}
and compatible with the spin connection, that is:
\begin{gather*}
X\<\varphi,\psi\>  = \<\na_X\varphi,\psi\>+\<\varphi,\na_X\psi\>,
\end{gather*}
\end{enumerate}
for all $X,Y\in\Gamma({\rm{T}}M)$ and $\psi,\varphi\in\Gamma(\Sigma M)$. We can f\/inally def\/ine a dif\/ferential operator acting on smooth spinor f\/ields, the Dirac operator, locally given by:
\[
\begin{array}{ccll}
D: & \Gamma(\Sigma M) & \longrightarrow & \Gamma(\Sigma M)\\
 & \psi & \longmapsto & D\psi=\sum\limits_{i=1}^ne_i\cdot\nabla_{e_i}\psi,
\end{array}
\]
 where $\{e_1,\dots,e_n\}$ is a local $g$-orthonormal frame. This f\/irst order dif\/ferential operator is elliptic and formally self adjoint.

\subsection{Conformal covariance of the Dirac operator}

We now focus on the conformal behavior of spinors on a Riemannian spin manifolds. We explain with some details the conformal covariance of the Dirac operator  and give an application of this property. For more details, we refer to \cite{hitchin:74,hijazi:86} or \cite{bourguignon.hijazi.milhorat.moroianu}. So consider a smooth function $u$ on the manifold $M$, and let $g_u=e^{2u}g$ be a conformal change of the metric. Then we have an obvious identif\/ication between the two $\SO_n$-principal bundles of $g$ and $g_u$-orthonormal oriented frames denoted respectively by $\SO M$ and $\SO_{u} M$. We can thus identify the corresponding $\Spin_n$-principal bundles $\Spin M$ and $\Spin_{u} M$, leading to a bundle isometry
\begin{gather*}
\begin{array}{ccc}\label{cc}
\Sigma M & \longrightarrow & \Sigma_uM \\
\varphi & \longmapsto & \varphi_u.
\end{array}
\end{gather*}
We can also relate the corresponding Levi-Civita connections, Clif\/ford multiplications and Hermitian scalar  products. Indeed, denoting by $\nabla^u$, $\cdot_u$ and $\<  \ ,\ \>_u$ the associated data which act on sections of the bundle $\Sigma_uM$, we can easily show that:
\begin{gather*}
\na^u_X\psi_u  =  \left(\na_X\psi-\frac{1}{2}X\cdot du\cdot\psi-\frac{1}{2}X(u)\psi\right)_u,\\
X_u\cdot_u\psi_u  =  (X\cdot\psi)_u,\\
\<\psi_u,\varphi_u\>_u  =  \<\psi,\varphi\>,
\end{gather*}
for all $\psi,\varphi\in\Gamma(\Sigma M)$, $X\in\Gamma({\rm{T}}M)$ and where $X_u:=e^{-u}X$ denotes the vector f\/ield over $(M^n,g_u)$ under the identif\/ication explained above. Using these identif\/ications, one can deduce the relation between the Dirac operators $D$ and $D_u$ acting respectively on sections of~$\Sigma M$ and~$\Sigma_uM$, that~is:
\begin{gather}\label{diraconforme}
D_u\psi_u=\left(e^{-\frac{n+1}{2}u}D(e^{\frac{n-1}{2}u}\psi)\right)_u.
\end{gather}
This formula clearly shows that the Dirac operator is a conformally covariant dif\/ferential opera\-tor of bidegree $(\frac{n-1}{2},\frac{n+1}{2})$.

\subsection{Eigenvalues of the Dirac operator}

A powerfull tool in the study of the Dirac operator is the Schr{\"o}dinger--Lichnerowicz formula which relates the square of the Dirac operator with the spinorial Laplacian. More precisely, we have:
\begin{gather*}
D^2=\na^*\na+\frac{1}{4}R,
\end{gather*}
where $\nabla^*$ is the $L^2$-formal adjoint of $\na$. An integration by parts using this formula leads to the following integral identity:
\begin{gather}\label{integrallichnerowicz}
\int_M|D\psi|^2dv=\int_M|\nabla\psi|^2dv+\frac{1}{4}\int_MR|\psi|^2dv
\end{gather}
for all $\psi\in\Gamma(\Sigma M)$. Combining this fundamental identity with the  Atiyah--Singer Index Theorem~\cite{atiyahpatodi} implies topological obstructions to  the existence of metrics with positive scalar curvature (see~\cite{lichnerowicz}). This vanishing result can be seen as a non-optimal estimate on the spectrum of the Dirac operator. In fact, if the scalar curvature is positive  and if $\psi$ is an eigenspinor associated with the f\/irst eigenvalue $\lambda_1(D)$ of the Dirac operator, then by (\ref{integrallichnerowicz}) one gets:
\begin{gather*}
\lambda_1^2(D)> \frac{1}{4}\,\underset{M}{\inf}(R).
\end{gather*}
Optimal eigenvalues estimate could be obtained by introducing the twistor operator $T$, which is the projection of $\nabla$ on the kernel of the Clif\/ford multiplication. It is locally given by:
\begin{gather*}
T_X\psi=\na_X\psi+\frac{1}{n}X\cdot D\psi
\end{gather*}
for all $\psi\in\Gamma(\Sigma M)$ and $X\in\Gamma({\rm{T}}M)$. Thus using the relation:
\begin{gather}\label{twistor}
|\na\psi|^2=|T\psi|^2+\frac{1}{n}|D\psi|^2,
\end{gather}
T.~Friedrich \cite{friedrich}  proved that the  f\/irst eigenvalue $\lambda_1(D)$ of the Dirac operator satisf\/ies:
\begin{gather*}
\lambda_1(D)^2\geq \frac{n}{4(n-1)}\,\underset{M}{\inf}(R)
\end{gather*}
with equality if and only if the eigenspinor associated with the f\/irst eigenvalue is a Killing spinor, that is for all $X\in\Gamma({\rm{T}}M)$:
\begin{gather}\label{killing}
\na_X\psi=-\frac{\lambda_1(D)}{n}X\cdot\psi .
\end{gather}

\subsection{Conformal lower bounds for the  eigenvalues of the Dirac operator}

We now prove the following result due to the f\/irst author:
\begin{theorem}[\cite{hijazi:86,hijazi:91}] \label{hijazi}
Let $(M^n,g)$ be a closed compact Riemannian spin manifold of dimension $n\geq 2$. Then the first eigenvalue of the Dirac operator satisfies:
\begin{gather}\label{hijazi1}
\lambda_1(D)^2\geq \frac{n}{4(n-1)}\,\underset{u}{\sup}\,\underset{M}{\inf}\,\big(R_ue^{2u}\big).
\end{gather}
Moreover, equality is achieved if and only if the eigenspinor associated with the eigenvalue $\lambda_1(D)$ is a Killing spinor. In particular, the manifold $(M^n,g)$ is Einstein.
\end{theorem}

\begin{proof}
Consider an eigenspinor $\psi\in\Gamma(\Sigma M)$ of the Dirac operator associated with the eigenvalue~$\lambda$. Now if we let $\varphi=e^{-\frac{n-1}{2}u}\psi\in\Gamma(\Sigma M)$ then the relation (\ref{diraconforme}) gives $D_u\varphi_u=\lambda e^{-u}\varphi_u$. Thus combining formulae~(\ref{integrallichnerowicz}) and (\ref{twistor}), we have:
\begin{gather*}
\frac{n-1}{n}\int_M|D_u\varphi_u|^2dv_u=\int_M|T^u\varphi_u|^2dv_u+\frac{1}{4}\int_MR_u|\varphi_u|^2dv_u,
\end{gather*}
which leads to:
\begin{gather*}
\frac{n-1}{n}\lambda^2\int_M e^{-2u}|\varphi_u|^2dv_u\geq\frac{1}{4}\int_MR_u|\varphi_u|^2dv_u\geq\frac{1}{4}\underset{M}{\inf}(R_u e^{2u})\int_Me^{-2u}|\varphi_u|^2dv_u
\end{gather*}
for all $u\in C^\infty(M)$. Inequality~(\ref{hijazi1}) follows directly. Suppose now that equality is achieved in~(\ref{hijazi1}), then we have:
\begin{gather*}
T^u_{X_u}\varphi_u=\na^u_{X_u}\varphi_u+\big(\lambda_1(D)/n\big) e^{-u} X_u\cdot_u\varphi_u=0,
\end{gather*}
for all $X\in\Gamma({\rm{T}}M)$. However, one can compute that in this case, the function $u$ has to be constant (see \cite{cours} for example) and thus $\psi\in\Gamma(\Sigma M)$ is a Killing spinor (that is it satisf\/ies equation~(\ref{killing})).
\end{proof}

We will now derive two results from the interpretation of the right-hand side of inequality~(\ref{hijazi1}). As we will see, these estimates will lead to some natural geometric invariants.

\subsubsection{The 2-dimensional case}\label{n=2}

We focus here on the case of compact closed surfaces and show that
\begin{gather*}
\lambda_1(D)^2\geq \frac{2\pi\,\chi(M^2)}{\aire}.
\end{gather*}
This has f\/irst been observed by B{\"a}r (see \cite{baer:92}).  For $n=2$, inequality~(\ref{hijazi1}) reads:
\begin{gather*}
\lambda_1(D)^2\geq \frac{1}{2}\,\underset{u}{\sup}\,\underset{M}{\inf}\,\big(R_ue^{2u}\big).
\end{gather*}
 First note that for all $u\in C^\infty(M)$:
\begin{gather*}
\underset{M}{\inf}\,\big(R_ue^{2u}\big)\leq\frac{1}{\aire}\int_M R_udv_u=\frac{1}{\aire}\int_M Rdv=\frac{4\pi\,\chi(M^2)}{\aire},
\end{gather*}
where we have used the fact that $R_u=2K_u$ and the Gauss--Bonnet formula~(\ref{gausssurface}). On the other hand, one can easily show the existence of a unique (up to an additive constant) smooth function $u_0\in C^\infty(M)$ such that:
\begin{gather*}
\Delta u_0=\frac{1}{\aire}\int_M K dv-K,
\end{gather*}
which proves the relation:
\begin{gather*}
\frac{1}{2}\,\underset{u}{\sup}\,\underset{M}{\inf}\,\big(R_ue^{2u}\big)=\frac{1}{2}\,R_{u_0}e^{2u_0}=\frac{2\pi\,\chi(M^2)}{\aire}.
\end{gather*}

\subsubsection[The case $n\geq 3$]{The case $\boldsymbol{n\geq 3}$}\label{n>2}

Here we will show that the right-hand side of inequality (\ref{hijazi1}) is given by the f\/irst eigenvalue~$\lambda_1(L)$ of the conformal Laplacian $L$ (see Section~\ref{laplacienconforme}), that is we get:
\begin{gather}\label{hijazi2}
\lambda_1(D)^2\geq\frac{n}{4(n-1)}\,\lambda_1(L).
\end{gather}
 First note that since $n\geq 3$, one can consider the conformal change of metrics def\/ined by $g_h=h^{\frac{4}{n-2}} g$ with $h$ a smooth positive function on $M$ and thus we have:
\begin{gather*}
R_u e^{2u}=R_h h^{\frac{4}{n-2}}=h^{-1}L h.
\end{gather*}
 Now we choose the conformal weight $h=h_1$ as being an eigenfunction of the conformal Laplacian associated with the f\/irst eigenvalue $\lambda_1(L)$. Such a function can be assumed to be positive and normalized (with unit $L^2$-norm) and satisf\/ies:
\begin{gather*}
\lambda_1(L)=4\frac{n-1}{n-2}\int_M h_1\Delta h_1dv+\int_M Rh^2_1dv.
\end{gather*}
On the other hand, for any $f$ smooth and positive function on $M$, we can write $h_1=f F$ with $F$ a smooth positive function on $M$. Using this expression in the preceding integral equality with an integration by parts leads to:
\begin{gather*}
\lambda_1(L)=\int_M\left(4\frac{n-1}{n-2}f^{-1}\Delta f+R\right)h_1^2dv+4\frac{n-1}{n-2}\int_Mf^2|\na F|^2dv,
\end{gather*}
then
\begin{gather*}
\lambda_1(L)\geq \underset{M}{\inf}\left(4\frac{n-1}{n-2}f^{-1}\Delta f+R\right)=\underset{M}{\inf}\big(f^{-1}L f\big)
\end{gather*}
for all $f$ smooth and positive. It is clear that the above inequality is achieved if and only if the function $f$ is an eigenfunction of $L$ associated with its f\/irst eigenvalue.  Using this inequality, we can compare the spectrum of the Dirac operator with a conformal invariant of the manifold $(M^n,g)$, {\it the Yamabe invariant} $Y(M^n,[g])$. This invariant appears naturally in the context of the Yamabe problem (see \cite{lee.parker:87} for example) and plays a fundamental role in its solution. It is def\/ined by:
\begin{gather*}
Y(M^n,[g])=\underset{f\in H^2_1\setminus\{0\}}{\inf}\frac{\int_{M}\big(4\frac{n-1}{n-2}|\na f|^2+R\,f^2\big)dv}{\Big(\int_{M}|f|^{\frac{2n}{n-2}}dv\Big)^{\frac{n-2}{2}}},
\end{gather*}
where $H^2_1$ denotes the space of $L^2$-integrable functions  as well as their f\/irst derivatives. Indeed, using the variational characterization of $\lambda_1(L)$ given by the Rayleigh quotient:
\begin{gather*}
\lambda_1(L)=\underset{f\in H^2_1\setminus\{0\}}{\inf}\frac{\int_{M}\big(4\frac{n-1}{n-2}|\na f|^2+R\,f^2\big)dv}{\int_{M}f^2dv},
\end{gather*}
and applying the H{\"o}lder inequality leads to $\lambda_1(L) \Vol(M^n,g)^{\frac{2}{n}}\geq Y(M^n,[g])$. Combining this estimate with Inequality~(\ref{hijazi2}) shows that:
\begin{gather*}
\lambda_1(D)^2 \Vol(M^n,g)^{\frac{2}{n}}\geq\frac{n}{4(n-1)}\,Y(M^n,[g]).
\end{gather*}

\section[First eigenvalues of conformally covariant differential operators]{First eigenvalues of conformally covariant\\ dif\/ferential operators}

In this section, we show that in dimension $4$ the total $Q$-curvature bounds from below the square of the Yamabe invariant. Using the conformal covariance of the Yamabe and the Paneitz--Branson operators together with a special choice of the conformal factor, we get a relation between their f\/irst eigenvalues. This combined  with the inequality  (\ref{hijazi2}) relate the Dirac, Yamabe and Paneitz--Branson operators through appropriate powers of their f\/irst eigenvalues. We show:

\begin{theorem}\label{main4}
Let $(M^4,g)$ be a closed compact $4$-dimensional Riemannian manifold. Then, the first eigenvalue of the Yamabe operator satisfies
\begin{gather}\label{Inequality4}
\lambda_1(L)^2\geq\frac{24}{\Vol(M^4,g)}\int_M Q dv.
\end{gather}
 Moreover, equality occurs if and only if $g$ is an Einstein metric.
\end{theorem}

\begin{proof} Recall that if $n=4$, the  $Q$-curvature is def\/ined by:
\begin{gather*}
6Q=R^2-3|\Ric|^2+\Delta R
\end{gather*}
and then one can easily check that:
\begin{gather*}
24 Q= R^2-12|E|^2+4\Delta R,
\end{gather*}
 where $E:=\Ric-(R/4)g$ is the Einstein tensor of $(M^4,g)$. Now for $g_u=e^{2u}g$, we can write:
\begin{gather}\label{dimension4}
24\int_MQ_udv_u=\int_M R^2_u dv_u-12\int_M |E_u|^2_u dv_u\leq \int_M R^2_u dv_u
\end{gather}
since $|E_u|^2_u\geq 0$. We now choose an adapted conformal weight, namely an eigenfunction of the Yamabe operator associated with its f\/irst eigenvalue that is a smooth positive function $h_1$ such that:
\begin{gather*}
L h_1=\lambda_1(L)\;h_1.
\end{gather*}
Consider the conformal change of metrics $g_{h_1}=h_1^2 g\in [g]$ and inequality (\ref{dimension4}) written for the metric $g_{h_1}$ reads:
\begin{gather*}
24\int_M Q_{h_1} dv_{h_1}\leq \int_M R^2_{h_1} dv_{h_1}=\lambda_1(L)^2\Vol(M^4,g)
\end{gather*}
where we used the fact that:
\begin{gather*}
R_{h_1}=h_1^{-3}L h_1=\lambda_1(L)h_1^{-2}
\end{gather*}
and $dv_{h_1}=h_1^4dv$. Finally since $n=4$, we use the conformal invariance of the left-hand side of the preceding inequality (see Section~\ref{paneitz}) to get inequality~(\ref{Inequality4}).
\end{proof}

If we now apply inequality (\ref{hijazi2}), we obtain:

\begin{corollary}\label{corr4}
Under the assumptions of Theorem~{\rm \ref{main4}}, if $M$ is spin and $\lambda_1(L)>0$, then:
\begin{gather*}
\lambda_1(D)^4 \geq\frac{1}{9}\lambda_1(L)^2\geq\frac{8}{3}\frac{\int_M Q dv}{\Vol(M^4,g)}.
\end{gather*}
Equality in both inequalities is characterized by the existence of a Killing spinor, in particular, the manifold is the round sphere.
\end{corollary}

\begin{remark}
By the Chern--Gauss--Bonnet formula, it follows that the total Branson's $Q$-curvature is a multiple of the total $\sigma_2$-curvature (the classical second elementary function associated with the Schouten tensor).  Hence Corollary~\ref{corr4} is exactly the result obtained by G.~Wang \cite{gwang} but it seems that there is a serious gap in his proof regarding the prescription, in a conformal class, of the $\sigma_2$-curvature function.
\end{remark}

We now consider the general case:

\begin{theorem}\label{main}
Let $(M^n,g)$ be a closed compact Riemannian manifold with $n\geq 5$. If the first eigenvalue $\lambda_1(P)$ of the Paneitz--Branson operator is positive, then we have:
\begin{gather}\label{main1}
\lambda^2_1(L)\geq\frac{16n(n-1)^2}{(n^2-4)(n-4)}\lambda_1(P) .
\end{gather}
 Moreover equality occurs if and only if $g$ is an Einstein metric.
\end{theorem}

\begin{proof} First note that the $Q$-curvature, def\/ined in (\ref{qcurvature}), can be written as:
\begin{gather*}
Q=\frac{n^2-4}{8n(n-1)^2}R^2-\frac{2}{(n-2)^2}|E|^2+\frac{1}{2(n-1)}\Delta R,
\end{gather*}
where $E:=\Ric-\frac{R}{n}g$ is the Einstein tensor of $(M^n,g)$. We now consider the  metric  $g_u=u^{\frac{4}{n-4}}g$ where $u$ is a smooth positive function. Stokes formula gives:
\begin{gather*}
\int_MQ_u\, dv_u  =  \int_M\left(\frac{n^2-4}{8n(n-1)^2}R_u^2-\frac{2}{(n-2)^2}|E_u|_u^2+\frac{1}{2(n-1)}\Delta_u R_u\right)dv_u\\
\phantom{\int_MQ_u\, dv_u }{}  =  \int_M\left(\frac{n^2-4}{8n(n-1)^2}R_u^2-\frac{2}{(n-2)^2}|E_u|_u^2\right)dv_u  \leq  \frac{n^2-4}{8n(n-1)^2}\int_MR_u^2\, dv_u.
\end{gather*}
On the other hand, with the help of the relation (\ref{paneitzconforme}) and since $dv_u=u^{\frac{2n}{n-4}}dv$, one can check that:
\begin{gather*}
\frac{n-4}{2}\int_M Q_u\, dv_u = \int_Mu Pu\, dv,
\end{gather*}
which gives:
\begin{gather}\label{inegalite}
\int_M u Pu\, dv\leq \frac{(n^2-4)(n-4)}{16n(n-1)^2}\int_M R_u^2\, dv_u
\end{gather}
for all $u$ smooth and positive on $M$. We will now see that a suitable choice of the conformal weight will lead to the desired inequality. We choose $h_1\in C^\infty(M)$ a smooth eigenfunction of the conformal Laplacian associated with its f\/irst eigenvalue, that is $L h_1=\lambda_1(L)h_1$. It is a standard fact that $h_1$ can be chosen to be positive on $M$. Let $g_{u_1}={u_1}^{\frac{4}{n-4}}g\in [g]$ where $u_1:=h_1^{\frac{n-4}{n-2}}$ is a~smooth positive function on $M$. Applying (\ref{inegalite}) in the metric $g_{u_1}$ leads to:
\begin{gather*}
\int_M u_1 P u_1\, dv\leq \frac{(n^2-4)(n-4)}{16n(n-1)^2}\int_M R_{u_1}^2\, dv_{u_1}.
\end{gather*}
 The choice of $u_1$ allows to compute that the scalar curvature of the manifold $(M^n,g_{u_1})$ is given~by:
\begin{gather}\label{scalar}
R_{u_1}=h_1^{-\frac{n+2}{n-2}}L(h_1)=\lambda_1(L) h_1^{-\frac{4}{n-2}}
\end{gather}
and thus $(\ref{inegalite})$ reads:
\begin{gather*}
\int_M u_1 P u_1\, dv\leq \frac{(n^2-4)(n-4)}{16n(n-1)^2}\lambda_1(L)^2\int_M h_1^{2\frac{n-4}{n-2}} \, dv=\frac{(n^2-4)(n-4)}{16n(n-1)^2}\lambda_1(L)^2\int_M u_1^2\, dv.
\end{gather*}
Inequality (\ref{main1}) follows directly from the variational characterization of $\lambda_1(P)$. If equality is achieved, then it is clear that the manifold $(M^n, g_{u_1})$ is Einstein. However, with the help of (\ref{scalar}) we easily conclude that $u_1$ has to be constant and thus $(M^n, g)$ is also an Einstein manifold.
\end{proof}

Inequality~(\ref{hijazi2}) and Theorem~\ref{main} then give:
\begin{corollary}\label{corr1}
Under the assumptions of Theorem {\rm \ref{main}}, if $M$ is spin and $\lambda_1(L)>0$, then:
\begin{gather*}
\lambda_1(D)^4\geq\frac{n^2}{16(n-1)^2}\lambda^2_1(L)\geq\frac{n^3}{(n^2-4)(n-4)}\lambda_1(P).
\end{gather*}
Equality in both inequalities is characterized by the existence of a Killing spinor, in particular the manifold is Einstein.
\end{corollary}

\subsection*{Acknowledgements}

We would like to thank the referees for their careful reading and suggestions.

\pdfbookmark[1]{References}{ref}
\LastPageEnding

\end{document}